\documentclass[a4paper,11pt]{article}
\usepackage{amsmath}
\usepackage{amssymb}
\usepackage{amsthm}
\usepackage{latexsym}
\usepackage{amsmath}
\usepackage{amsthm}
\usepackage{graphicx}
\usepackage{txfonts}
\usepackage{bm}
\usepackage{epic,eepic}

\usepackage{geometry}

\newtheorem{theorem}{Theorem}[section]
\newtheorem{proposition}[theorem]{Proposition}

\newtheorem{lemma}[theorem]{Lemma}
\newtheorem{remark}{Remark}[section]

\numberwithin{equation}{section}

\newcommand{\todaye}{\the\year/\the\month/\the\day}

\newcommand{\RR}{\mathbb{R}}
\newcommand{\dom}{{\rm dom\,}}

\newcommand{\argmax}{\arg \max}

\title{Multiple Exchange Property for 
\\
M$\sp{\natural}$-concave Functions and Valuated Matroids}

\author{Kazuo Murota
\\
School of Business Administration, Tokyo Metropolitan University,
\\
Tokyo 192-0397, Japan; 
murota@tmu.ac.jp
}

\date{August 25, 2016 / April 11, 2017}
\begin{document}

\maketitle

\begin{abstract}
The multiple exchange property for matroid bases
is generalized for valuated matroids and
M$\sp{\natural}$-concave set functions.
The proof is based on the Fenchel-type duality theorem in discrete convex analysis.
The present result has an implication in economics:
The strong no complementarities (SNC) condition of 
Gul and Stacchetti   
is in fact equivalent to the gross substitutes (GS) condition of 
Kelso and Crawford.
\end{abstract}


\section{Introduction}
\label{SCintro}

Discrete convex analysis \cite{Fuj05,Mdcasiam} offers
a general framework of discrete optimization,
combining the ideas from submodular/matroid theory and convex analysis.
It has found applications in many different areas \cite{Mbonn09}, 
including systems analysis \cite{Mspr2000},
inventory theory in operations research \cite{SCB14},
and mathematical economics and game theory
\cite{FujTam07,ISST15,KTY14,MTcompeq03}.
The interaction between discrete convex analysis and mathematical economics
was initiated by \cite{DKM01} (see also \cite[Chapter 11]{Mdcasiam})
and accelerated by the crucial observation of Fujishige--Yang \cite{FY03}
that M$\sp{\natural}$-concavity 
(see Section \ref{SCresult} for the definition)
is equivalent to the gross substitutability (GS) of Kelso--Crawford \cite{KC82};
\cite{ST15jorsj} is a detailed recent survey on the relation between
M$\sp{\natural}$-concavity and (GS).

In matroid theory, 
one of the classical results \cite{Bry73,Gre73,Woo74} says that the basis family
of a matroid enjoys the multiple exchange property:
For two bases $X$ and $Y$ in a matroid and a subset $I \subseteq X \setminus Y$,
there exists a subset $J \subseteq Y \setminus X$ such that
$(X \setminus I) \cup J$ and $(Y \setminus J) \cup I$ are both bases.
As a quantitative version of this, we may naturally consider
the multiple exchange property for a set function $f$:
For two subsets $X, Y$ and a subset $I \subseteq X \setminus Y$,
there exists $J \subseteq Y \setminus X$ such that
$f( X) + f( Y )   \leq  f((X \setminus I) \cup J) +f((Y \setminus J) \cup I)$.

The objective of this paper is to 
establish this multiple exchange property
for M$\sp{\natural}$-concave functions and valuated matroids.
The results are described in Section \ref{SCresult} 
and the proof,
based on the Fenchel-type duality theorem in discrete convex analysis,
is given in Section \ref{SCproof}. 
Our result settles an old question in economics:
The strong no complementarities (SNC) condition of 
Gul and Stacchetti \cite{GS99}
is in fact equivalent to the gross substitutes condition.
This is discussed in Section \ref{SCecoimpl}.
Section \ref{SCproofconverse} offers a proof 
to the fact that the multiple exchange property characterizes 
M$\sp{\natural}$-concavity.

\section{Results}
\label{SCresult}

Let $N$ be a finite set, say,
$N = \{ 1,2,\ldots, n \}$.
For a function
$f: 2\sp{N} \to \RR \cup \{ -\infty \}$,
$\dom f$ denotes the effective domain of $f$, i.e.,
$\dom f = \{ X \mid f(X) > -\infty \}$.

A function
$f: 2\sp{N} \to \RR \cup \{ -\infty \}$ 
with $\dom f \not= \emptyset$
is called {\em M$\sp{\natural}$-concave}
\cite{Mdcasiam, MS99gp}, if,
for any $X, Y \in \dom f$ 
and $i \in X \setminus Y$,
we have  (i)
$X - i \in \dom f$, $ Y + i \in \dom f$ and
\begin{equation}  \label{mconcav1}
f( X) + f( Y ) \leq f( X - i ) + f( Y + i ),
\end{equation}
or (ii) there exists some $j \in Y \setminus X$ such that
$X - i +j \in \dom f$, $ Y + i -j  \in \dom f$ and
\begin{equation}  \label{mconcav2}
f( X) + f( Y ) \leq 
 f( X - i + j) + f( Y + i -j).
\end{equation}
Here we use short-hand notations
$X - i = X \setminus  \{ i \}$,
$Y + i = Y \cup \{ i \}$,
$X - i + j =(X \setminus  \{ i \}) \cup \{ j \}$,
and
$Y + i - j =(Y \cup \{ i \}) \setminus \{ j \}$.
This property
is referred to as the {\em exchange property}.
The exchange property 
can be expressed more compactly as:
\begin{quote}
\mbox{\bf (M$\sp{\natural}$-EXC)} [Exchange property] \ 
For any $X, Y \subseteq N$ 
and $i \in X \setminus Y$, we have
\begin{align}
f( X) + f( Y )   &\leq 
   \max\left[ f( X - i ) + f( Y + i ),  
 \right.
\notag \\
    &    \left.        \phantom{\max_{j \in Y \setminus X}} 
 \max_{j \in Y \setminus X}  \{ f( X - i + j) + f( Y + i -j) \}\right] ,
\label{mnatconcavexc2}
\end{align}
\end{quote}
where
$(-\infty) + a = a + (-\infty) = (-\infty) + (-\infty)  = -\infty$ for $a \in \RR$,
$-\infty \leq -\infty$, and
a maximum taken over an empty set
is defined to be $-\infty$.

In this paper we are concerned with the {\em multiple exchange property}:
\begin{quote}
\mbox{\bf (M$\sp{\natural}$-EXC$_{\rm\bf m}$)} [Multiple exchange property]%
\footnote{
In economics \cite{GS99}, the multiple exchange property (M$\sp{\natural}$-EXC$_{\rm m}$) 
is called 
``strong no complementarities property (SNC)''.
See Section \ref{SCecoimpl}.
}
\ 
For any $X, Y \subseteq N$ and $I \subseteq X \setminus Y$,
there exists $J \subseteq Y \setminus X$ such that
$ 
f( X) + f( Y )   \leq  f((X \setminus I) \cup J) +f((Y \setminus J) \cup I) 
$, 
i.e., 
\begin{align}
f( X) + f( Y )   \leq 
 \max_{J \subseteq Y \setminus X} 
 \{  f((X \setminus I) \cup J) +f((Y \setminus J) \cup I)  \}.
\label{mnatconcavexcmult}
\end{align}
\end{quote}

\begin{theorem} \label{THmultexchmnat}
An M$\sp{\natural}$-concave function
$f: 2\sp{N} \to \RR \cup \{ -\infty \}$ 
has the multiple exchange property 
{\rm (M$\sp{\natural}$-EXC$_{\rm m}$)}.
\end{theorem}
\begin{proof}
The proof is given in Section \ref{SCproof}.
\end{proof}

As an immediate corollary we obtain the multiple exchange property 
for the maximizers.
The set of the maximizers of $f$ is denoted by $\argmax f$.

\begin{theorem} \label{THmultexchmnatmax}
Let $f: 2\sp{N} \to \RR \cup \{ -\infty \}$ 
be an M$\sp{\natural}$-concave function.
For any  $X, Y \in \argmax f$ and $I \subseteq X \setminus Y$,
there exists a subset 
$J \subseteq Y \setminus X$
such that
$(X \setminus I) \cup J \in \argmax f$ and $(Y \setminus J) \cup I \in \argmax f$.
\end{theorem}

The concept of valuated matroid due to 
of Dress--Wenzel \cite{DW90,DW92} (see also \cite[Chapter 5]{Mspr2000})
is defined in terms of an exchange property similar to (M$\sp{\natural}$-EXC).
A function
$f: 2\sp{N} \to \RR \cup \{ -\infty \}$ 
with $\dom f \not= \emptyset$
is called a {\em valuated matroid}, if,
for any $X, Y \subseteq N$ and $i \in X \setminus Y$,
it holds that
\begin{align}
f( X) + f( Y )   \leq 
 \max_{j \in Y \setminus X}  \{ f( X - i + j) + f( Y + i -j) \} .
\label{valmatexc}
\end{align}
A valuated matroid is nothing but 
an M$\sp{\natural}$-concave function $f$ such that
$\dom f$ consists of equi-cardinal subsets, i.e.,
$|X| = |Y|$ for any $X, Y \in \dom f$.
In this case, $\dom f$ forms the basis family of a matroid on $N$.

As a corollary of Theorem~\ref{THmultexchmnat} we obtain the following.

\begin{theorem} \label{THmultexchvalmat}
A valuated matroid $f$ has the multiple exchange property 
{\rm (M$\sp{\natural}$-EXC$_{\rm m}$)} with $|J|=|I|$.
\end{theorem}

This theorem contains, as a special case,
the multiple exchange theorem for matroid bases 
due to Brylawski \cite{Bry73}, Greene \cite{Gre73} and Woodall \cite{Woo74};
see also \cite{Kun86b,McDia75,Sch03}.

\begin{theorem}[\protect{\cite{Bry73,Gre73,Woo74}}] \label{THmultexchbase}
Let $X$ and $Y$ be bases in a matroid, and let 
$I \subseteq X \setminus Y$.
Then there exists a subset 
$J \subseteq Y \setminus X$
such that
$(X \setminus I) \cup J$ and $(Y \setminus J) \cup I$ are both bases.
\end{theorem}

The converse of Theorem~\ref{THmultexchmnat} should be intuitively obvious, 
but a formal proof is needed.
We have to assure that 
for $I=\{ i \}$ in {\rm (M$\sp{\natural}$-EXC$_{\rm m}$)}
there exists $J$ with $|J| \leq 1$. 

\begin{proposition} \label{PRmultTOusual}
{\rm (M$\sp{\natural}$-EXC$_{\rm m}$)} implies
{\rm (M$\sp{\natural}$-EXC)}.
\end{proposition}
\begin{proof}
We provide two proofs in this paper.
The first, given in Section \ref{SCecoimpl} for an economic implication,
is rather indirect, 
relying essentially on some (generalizations of) known results in economics.
The second proof, given in Section \ref{SCproofconverse},
 is more straightforward, 
relying on basic results in matroid theory and discrete convex analysis.
\end{proof}

{}From Proposition \ref{PRmultTOusual}
and Theorem~\ref{THmultexchmnat}
we can obtain a characterization of 
M$\sp{\natural}$-concave functions
in terms of the multiple exchange property.

\begin{theorem} \label{THmnatiffmultexch}
A function
$f: 2\sp{N} \to \RR \cup \{ -\infty \}$ 
is M$\sp{\natural}$-concave 
if and only if it has the multiple exchange property 
{\rm (M$\sp{\natural}$-EXC$_{\rm m}$)}.
\end{theorem}

\section{Proof of Theorem~\ref{THmultexchmnat}}
\label{SCproof}

In this section we give a proof to the main theorem, Theorem~\ref{THmultexchmnat}.
Let 
$f: 2\sp{N} \to \RR \cup \{ -\infty \}$ 
be an M$\sp{\natural}$-concave function,
$X, Y \in \dom f$ and $I \subseteq X \setminus Y$.
We shall prove
\begin{align}
f( X) + f( Y )   \leq 
 \max_{J \subseteq Y \setminus X} 
 \{  f((X \setminus I) \cup J) +f((Y \setminus J) \cup I)  \} .
\label{mnatconcavexcmult2}
\end{align}
Our proof is based on the Fenchel-type duality theorem in discrete convex analysis.

With the notations
\begin{align*}
 &C = X \cap Y,
\qquad
 X_{0} = X \setminus Y = X \setminus C,
\qquad
 Y_{0} = Y \setminus X = Y \setminus C ,
\\
& f_{1}(J) =  f((X \setminus I) \cup J) 
  = f( (X_{0} \setminus I) \cup C \cup J)
\qquad (J \subseteq Y_{0}),
\\
& f_{2}(J) = f((Y \setminus J) \cup I)
  = f(  I \cup C \cup (Y_{0} \setminus J) )
\qquad (J \subseteq Y_{0}),
\end{align*}
the inequality (\ref{mnatconcavexcmult2}) is rewritten as 
\begin{align}
f( X) + f( Y )   \leq 
 \max_{J \subseteq Y_{0}}  \{ f_{1}(J) + f_{2}(J) \}.
\label{mnatconcavexcmult3}
\end{align}
Both  $f_{1}$ and $f_{2}$
are M$\sp{\natural}$-concave set functions on $Y_{0}$.

Consider the (convex) conjugate functions of $f_{1}$ and $f_{2}$
given by
\begin{align*}
g_{1}(q) &=  
 \max_{J \subseteq Y_{0}} \{  f_{1}(J) - q(J) \} 
\qquad (q \in \RR\sp{Y_{0}}),
\\
g_{2}(q) &= 
 \max_{J \subseteq Y_{0}} \{  f_{2}(J) - q(J) \} 
\qquad (q \in \RR\sp{Y_{0}}),
\end{align*}
where $q(J) = \sum_{j \in J} q_{j}$.
For any  $J \subseteq Y_{0}$
and $q \in \RR\sp{Y_{0}}$, we have
\begin{align*}
f_{1}(J) + f_{2}(J) &=  
( f_{1}(J) - q(J)) + ( f_{2}(J) + q(J) )
\\  
 & \leq  
 \max_{J \subseteq Y_{0}} \{  f_{1}(J) - q(J) \} 
 + 
 \max_{J \subseteq Y_{0}} \{  f_{2}(J) + q(J) \} 
\\
 &=  g_{1}(q) + g_{2}(-q).
\end{align*}
The Fenchel-type duality theorem
in discrete convex analysis
\cite[Theorem 8.21 (1)]{Mdcasiam}
asserts that there exist $J$ and $q$ 
for which the above inequality holds in equality, i.e., 
\begin{align}
 \max_{J \subseteq Y_{0}}  \{ f_{1}(J) + f_{2}(J) \}
 = \min_{q \in \RR\sp{Y_{0}}} \{ g_{1}(q) + g_{2}(-q) \}.
\label{fencmaxmin}
\end{align}
Note that $\dom g_{1} = \dom g_{2} = \RR\sp{Y_{0}}$ 
and the assumption in 
\cite[Theorem 8.21 (1)]{Mdcasiam}
is satisfied.

The desired inequality (\ref{mnatconcavexcmult3})
follows from (\ref{fencmaxmin}) and Lemma \ref{LMg1qg2q} below.

\begin{lemma} \label{LMg1qg2q}
For any $q \in \RR\sp{Y_{0}}$, we have
$g_{1}(q) + g_{2}(-q) \geq f( X) + f( Y )$.
\end{lemma}
\begin{proof}
Let $g$ be the (convex) conjugate function of $f$, i.e.,
\begin{align*}
g(p) &=  
 \max_{Z \subseteq N} \{  f(Z) - p(Z) \} 
\qquad (p \in \RR\sp{N}).
\end{align*}
By the conjugacy theorem 
in discrete convex analysis
(\cite[Theorem 3.4]{Mbonn09}, \cite[Theorems 8.4, (8.10)]{Mdcasiam}),
$g$ is a polyhedral L$\sp{\natural}$-convex function.
In particular, it is submodular:
\begin{equation} \label{gsubm}
  g(p) + g(p') \geq g(p \vee p') + g(p \wedge p')
\qquad (p, p' \in \RR\sp{N}) ,
\end{equation}
where $p \vee p'$ and $p \wedge p'$ denote the vectors of component-wise
maximum and minimum, i.e.,
\[ 
 (p \vee p')_{i}   = \max(p_{i}, p'_{i}),
\qquad
 (p \wedge p')_{i} = \min(p_{i}, p'_{i}) .
\] 

For a vector $q \in \RR\sp{Y_{0}}$ we define 
$p\sp{(1)}, p\sp{(2)} \in \RR\sp{N}$ by
\begin{align*} 
p\sp{(1)}_{i}  =
   \left\{  \begin{array}{ll}
    q_{i}          &   (i  \in Y_{0}) ,     \\
   - M     &   (i \in X_{0} \setminus I ),  \\
   + M     &   (i \in I ),  \\
   - M     &   (i \in C ),  \\
   + M     &   (i \in N \setminus (X \cup Y)  ) , \\
                     \end{array}  \right.
\quad 
p\sp{(2)}_{i}  =
   \left\{  \begin{array}{ll}
    q_{i}          &   (i  \in Y_{0}) ,     \\
   + M     &   (i \in X_{0} \setminus I ) , \\
   - M     &   (i \in I ),  \\
   - M     &   (i \in C ),  \\
   + M     &   (i \in N \setminus (X \cup Y)  ) , \\
                     \end{array}  \right.
\end{align*}
where $M$ is a sufficiently large positive number.
Then we have
\begin{align*}
g_{1}(q) &=  
 \max_{J \subseteq Y_{0}} \{ f( (X_{0} \setminus I) \cup C \cup J) - q(J) \} 
\\ & =
g(p\sp{(1)}) - M (|X_{0} \setminus I|+|C|),
\\
g_{2}(-q) &= 
 \max_{J \subseteq Y_{0}} \{  
 f(  I \cup C \cup (Y_{0} \setminus J) ) + q(J) \} 
\\
 &= 
 \max_{K \subseteq Y_{0}} \{  
 f(  I \cup C \cup K ) - q(K) \} + q(Y_{0}) 
\\ & =
g(p\sp{(2)}) - M (|I|+|C|) + q(Y_{0}). 
\end{align*}
By adding these two and using submodularity (\ref{gsubm}) of $g$,
we obtain
\begin{align}
 g_{1}(q) + g_{2}(-q)
 &= 
g(p\sp{(1)}) +g(p\sp{(2)}) - M (|X|+|C|) + q(Y_{0})
\notag \\
 &\geq 
g(p\sp{(1)} \vee p\sp{(2)}) +g(p\sp{(1)} \wedge p\sp{(2)})
 - M (|X|+|C|) + q(Y_{0}). 
\label{g1g2gg}
\end{align}
Since
\begin{align*} 
(p\sp{(1)} \vee p\sp{(2)})_{i}   =
   \left\{  \begin{array}{ll}
    q_{i}          &   (i  \in Y_{0}) ,     \\
   + M     &   (i \in X_{0} \setminus I ),  \\
   + M     &   (i \in I ),  \\
   - M     &   (i \in C ),  \\
   + M     &   (i \in N \setminus (X \cup Y)  ) , \\
                     \end{array}  \right.
\quad 
(p\sp{(1)} \wedge p\sp{(2)})_{i}   =
   \left\{  \begin{array}{ll}
    q_{i}          &   (i  \in Y_{0}) ,     \\
   - M     &   (i \in X_{0} \setminus I ) , \\
   - M     &   (i \in I ),  \\
   - M     &   (i \in C ),  \\
   + M     &   (i \in N \setminus (X \cup Y)  ) , \\
                     \end{array}  \right.
\end{align*}
we have
\begin{align} 
g(p\sp{(1)} \vee p\sp{(2)}) & \geq
f(Y) - q(Y_{0}) + M|C|,
\label{gpveep}
\\
g(p\sp{(1)} \wedge p\sp{(2)}) & \geq
f(X) + M |X|.
\label{gpwedgep}
\end{align}
The substitution of  
(\ref{gpveep}) and (\ref{gpwedgep}) 
into (\ref{g1g2gg}) yields the desired inequality
\[
g_{1}(q) + g_{2}(-q) \geq f(X) + f(Y).
\]
\vspace{-2\baselineskip}\\
\end{proof}

\begin{remark}  \rm \label{RMmatroidbases}
Among several different proofs known 
for the multiple exchange property of matroid bases (Theorem~\ref{THmultexchbase}),
the proofs of Woodall \cite{Woo74} and McDiarmid \cite{McDia75}
are based on 
minimax duality formulas for matroid rank functions
(matroid union/intersection theorems).
Our proof of Theorem~\ref{THmultexchmnat} generalizes this idea
to M$\sp{\natural}$-concave functions.
Note that the matroid union/intersection theorems 
are special cases of 
the Fenchel-type duality theorem
for M$\sp{\natural}$-concave functions 
\cite[Section 8.2.3]{Mdcasiam}, \cite[Section 5]{Mrims10}.
\end{remark}

\begin{remark}  \rm \label{RMalgforJ}
The above proof shows that
the subset $J$ in (M$\sp{\natural}$-EXC$_{\rm m}$) 
can be computed in polynomial time
by an adaptation of the valuated matroid intersection algorithm
\cite[Chapter 5]{Mspr2000}.
\end{remark}

\section{An Implication in Economics}
\label{SCecoimpl}

For a vector $p \in \RR\sp{N}$ we define
\begin{equation} \label{Dpdef}
 D(p | f) = \argmax_{X} \{ f(X) - p(X) \mid X \subseteq N \},
\end{equation}
where $p(X) = \sum_{i \in X} p_{i}$.
In economic applications
where $f$ denotes a utility (valuation) function over indivisible goods, 
$p$ is interpreted as the vector of prices and 
$D(p) = D(p | f)$ represents the demand correspondence.
We use notation $f[-p]$ for the function defined by
$f[-p](X) = f(X) - p(X)$ for $X \subseteq N$. 

Kelso and Crawford \cite{KC82}
introduced the following property for
$f: 2\sp{N} \to \RR \cup \{ -\infty \}$,
which turned out to be the key property 
in discussing economies with indivisible goods%
\footnote{
To be precise, Kelso and Crawford \cite{KC82}
as well as Gul and Stacchetti \cite{GS99} 
and Fujishige and Yang \cite{FY03}
treat the case of $f: 2\sp{N} \to \RR$,
with $\dom f = 2\sp{N}$.
It can be verified that (\ref{econImp1}) is true for 
$f: 2\sp{N} \to \RR \cup \{ -\infty \}$.
}:  
\begin{quote}
 {\bf (GS)} [Gross Substitutes property] \ 
For any vectors $p$ and $q$ 
with $p \leq q$ and $X \in D(p | f)$,
 there exists $Y \in D(q | f)$ such that
$ \{ i \in X \mid p_{i} = q_{i} \} \subseteq Y$.
\end{quote}
Gul and Stacchetti \cite{GS99} considered
the following three properties:
\begin{quote}
{\bf (SI)} [Single Improvement property] \ 
For any $p \in \RR\sp{N}$, if
$X \not\in D(p | f)$, 
 there exists $Y \subseteq N$ such that
$|X \setminus Y| \leq 1$, $|Y \setminus X| \leq 1$, 
and $f[-p](X) < f[-p](Y)$,
\end{quote}
\begin{quote}
{\bf (NC)} [No Complementarities property] \ 
For any $p \in \RR\sp{N}$, if
$X, Y \in D(p | f)$ and $I \subseteq X \setminus Y$,
there exists $J \subseteq Y \setminus X$ such that
$(X \setminus I) \cup J \in D(p | f)$,
\end{quote}
\begin{quote}
{\bf (SNC)}  [Strong No Complementarities property] \ 
For $X, Y \subseteq N$ and $I \subseteq X \setminus Y$,
there exists $J \subseteq Y \setminus X$ such that
$
f(X)+f(Y) \leq
f((X \setminus I) \cup J) +f((Y \setminus J) \cup I)
$.
\end{quote}
They showed that 
(NC) and (SI) are equivalent to (GS), and these (mutually equivalent) conditions 
are implied by (SNC).
Subsequently, 
Fujishige and Yang \cite{FY03} pointed out that
(GS) is equivalent to (M$\sp{\natural}$-EXC) for M$\sp{\natural}$-concavity.
These results are summarized schematically here as:
\begin{equation} \label{econImp1}
\mbox{(SNC)} \Longrightarrow \mbox{(NC)} 
\iff \mbox{(GS)} \iff  \mbox{(SI)} \iff  \mbox{(M$\sp{\natural}$-EXC)} .
\end{equation}
Since (SNC) and (M$\sp{\natural}$-EXC$_{\rm m}$) 
are mathematically the same, 
and (M$\sp{\natural}$-EXC$_{\rm m}$) follows from (M$\sp{\natural}$-EXC) 
by Theorem~\ref{THmultexchmnat}, 
we now see that
the above five properties are in fact equivalent:
\begin{equation} \label{econImp2}
\mbox{(SNC)} \iff \mbox{(NC)} 
\iff \mbox{(GS)} \iff  \mbox{(SI)} \iff  \mbox{(M$\sp{\natural}$-EXC)} .
\end{equation}

In this context it would be natural to consider the following 
simultaneous version of (NC):
\begin{quote}
{\bf (NCsim)} 
For any $p \in \RR\sp{N}$, if
$X, Y \in D(p | f)$ and $I \subseteq X \setminus Y$,
there exists $J \subseteq Y \setminus X$ such that
$(X \setminus I) \cup J \in D(p | f)$ and $(Y \setminus J) \cup I \in D(p | f)$.
\end{quote}
Obviously,
$\mbox{(SNC)} \Longrightarrow   \mbox{(NCsim)}$
and 
$\mbox{(NCsim)} \Longrightarrow  \mbox{(NC)}$.
Hence (NCsim) is also equivalent to (GS).

We conclude this section by stating the equivalence of all the six properties as a theorem.

\begin{theorem} \label{THecomequiv}
For a function 
$f: 2\sp{N} \to \RR \cup \{ -\infty \}$,
we have the following equivalence:
\[ 
\mbox{\rm (M$\sp{\natural}$-EXC$_{\rm m}$)}=
\mbox{\rm (SNC)} \iff \mbox{\rm (NCsim)} \iff \mbox{\rm (NC)} 
\iff \mbox{\rm (GS)} \iff  \mbox{\rm (SI)} \iff  \mbox{\rm (M$\sp{\natural}$-EXC)} .
\] 
\end{theorem}

\section{A Direct Proof of Proposition \ref{PRmultTOusual}}
\label{SCproofconverse}

The proof of
``{\rm (M$\sp{\natural}$-EXC$_{\rm m}$)} $\Longrightarrow$ {\rm (M$\sp{\natural}$-EXC)}''
consists of the following propositions,
which refer to the following conditions
for a set family $\mathcal{F} \subseteq 2\sp{N}$:

\begin{quote}
\mbox{\bf (B$\sp{\natural}$-EXC)} 
For any $X, Y \in \mathcal{F}$ and $i \in X \setminus Y$, \ 
we have
(i) $X - i \in \mathcal{F}$, $ Y + i \in \mathcal{F}$ 
\ or \  
\\
(ii) there exists some $j \in Y \setminus X$ such that
$X - i +j \in \mathcal{F}$, $ Y + i -j  \in \mathcal{F}$,
\end{quote}

\begin{quote}
\mbox{\bf (B$\sp{\natural}$-EXC$_{\rm m}$)} 
For any $X, Y \in \mathcal{F}$ and $I \subseteq X \setminus Y$,
there exists $J \subseteq Y \setminus X$ such that
$(X \setminus I) \cup J \in \mathcal{F}$ and $(Y \setminus J) \cup I \in \mathcal{F}$.
\end{quote}

\begin{proposition} \label{PRmultdom} 
If $f$ satisfies {\rm (M$\sp{\natural}$-EXC$_{\rm m}$)}, then
$\dom f$ satisfies {\rm (B$\sp{\natural}$-EXC$_{\rm m}$)}.
\end{proposition}
\begin{proof}
This is obviousD
\end{proof}

\begin{proposition} \label{PRmultexchset}
For a set family $\mathcal{F}$,
{\rm (B$\sp{\natural}$-EXC$_{\rm m}$)} $\Longrightarrow$ {\rm (B$\sp{\natural}$-EXC)}.
\end{proposition}
\begin{proof}
The proof is given in Section \ref{SCgexcmTOgexc0}.
\end{proof}

\begin{proposition} \label{PRmultTOusualDomCond}
If $\dom f$ satisfies {\rm (B$\sp{\natural}$-EXC)}, then
{\rm (M$\sp{\natural}$-EXC$_{\rm m}$)} $\Longrightarrow$ {\rm (M$\sp{\natural}$-EXC)}.
\end{proposition}
\begin{proof}
The proof is given in Section \ref{SCmnatmTOmnat0}.
\end{proof}

\subsection{Proof of Proposition \ref{PRmultexchset}}
\label{SCgexcmTOgexc0}

We make use of the following proposition
(Tardos \cite[Theorem 2.3]{Tar85}, Murota--Shioura \cite[Remark 5.2]{MS99gp}),
where
\begin{quote}
\mbox{\bf (B$\sp{\natural}$-EXC$_{\pm}$)} 
For any $X, Y \in \mathcal{F}$ and $i \in X \setminus Y$, \ 
both (a) and (b) hold, where
\\ \quad
(a) (i)
 $X - i \in \mathcal{F}$
\ or \ (ii) 
$X - i +j \in \mathcal{F}$
for some $j \in Y \setminus X$;
\\ \quad
(b) (i)
$ Y + i \in \mathcal{F}$ 
\ or \ (ii) 
$ Y + i - k  \in \mathcal{F}$
for some $k \in Y \setminus X$.
\end{quote}

\begin{proposition}
\label{PRgmattardos}
The following three conditions are equivalent.

{\rm (i)}
$\mathcal{F}$ is a generalized matroid.

{\rm (ii)}
$\mathcal{F}$ satisfies {\rm (B$\sp{\natural}$-EXC$_{\pm}$)}.

{\rm (iii)}
$\mathcal{F}$ satisfies {\rm (B$\sp{\natural}$-EXC)}.
\end{proposition}

With this equivalence, 
the statement
``{\rm (B$\sp{\natural}$-EXC$_{\rm m}$)} $\Longrightarrow$ {\rm (B$\sp{\natural}$-EXC)}''
in Proposition \ref{PRmultexchset}
is immediate from the following.

\begin{proposition} \label{PRgexcmTOgexc0}
{\rm (B$\sp{\natural}$-EXC$_{\rm m}$)} $\Longrightarrow$ {\rm (B$\sp{\natural}$-EXC$_{\pm}$)}.
\end{proposition}
\begin{proof}
We prove by contradiction.
Suppose that {\rm (B$\sp{\natural}$-EXC$_{\pm}$)}
fails for some $X, Y \in \mathcal{F}$,
and take such $(X,Y)$ with
$|X \Delta Y | = | X \setminus Y | + | Y \setminus X | $
minimum.
There exists $i \in X \setminus Y$ such that

\begin{quote}
(F) : \quad   (a) or (b)  fails for $(X,Y,i)$
\ \ (including: both (a) and (b) fail).
\end{quote}

By (B$\sp{\natural}$-EXC$_{\rm m}$)
there exists $J \subseteq Y \setminus X$ such that
$(X -i) \cup J \in \mathcal{F}$ and $(Y \setminus J) + i \in \mathcal{F}$.
Choose such $J$ with $|J|$ minimum.
If $|J| = 0$, we have (a-i) and (b-i), which contradicts (F).
If $|J| = 1$, say, $J = \{ j \}$, we have (a-ii) and (b-ii), which contradicts (F).

Suppose that $|J| \geq 2$.  
Take $j \in J$.

\begin{itemize}
\item
(B$\sp{\natural}$-EXC$_{\rm m}$) for $((X -i) \cup J, X, j)$ yields
\\ \  
[Case X0: $(X -i) \cup (J-j)  \in \mathcal{F}$ and $X  + j \in \mathcal{F}$]
or  
[Case X1: $X \cup (J-j)  \in \mathcal{F}$ and $X -i + j \in \mathcal{F}$].

\item
(B$\sp{\natural}$-EXC$_{\rm m}$) for $(Y, (Y \setminus J) + i,j)$ yields
\\  \ 
[Case Y0: $Y - j  \in \mathcal{F}$ and $Y \setminus (J-j) + i\in \mathcal{F}$]
or
[Case Y1: $Y + i - j  \in \mathcal{F}$ and $Y \setminus (J-j) \in \mathcal{F}$].
\end{itemize}

We have four cases to consider:
X0--Y0, X1--Y1, X1--Y0 and X0--Y1.

\begin{itemize}
\item
Case X0--Y0:
We have $(X -i) \cup (J-j)  \in \mathcal{F}$ and $Y \setminus (J-j) + i\in \mathcal{F}$, contradicting the minimality of $|J|$.

\item
Case X1--Y1:
We have $X -i + j \in \mathcal{F}$ and $Y + i - j  \in \mathcal{F}$,
and hence (a-ii) and (b-ii). 
A contradiction to (F).

\item
Case X1--Y0:
By Case X1, we have $X - i + j  \in \mathcal{F}$, i.e., (a-ii).
Let $X'=X \cup (J-j) \in \mathcal{F}$ and note that
$|X' \Delta Y | < |X \Delta Y |$ by $|J| \geq 2$.
Then $(X',Y)$ must satisfy {\rm (B$\sp{\natural}$-EXC$_{\pm}$)}.
Since $i \in X' \setminus Y$, 
{\rm (B$\sp{\natural}$-EXC$_{\pm}$)} for $(X',Y,i)$
shows that
(b-i)
$ Y + i \in \mathcal{F}$ 
\ or \ (b-ii) 
$ Y + i - k  \in \mathcal{F}$
for some $k \in Y \setminus X'$.
Since $k \in Y \setminus X' \subseteq Y \setminus X$,
this means that (b) holds for $(X,Y,i)$.
A contradiction to (F).

\item
Case X0--Y1:
(A similar argument as in Case X1--Y0):
By Case Y1, we have $Y + i - j  \in \mathcal{F}$, i.e., (b-ii).
Let $Y'=Y \setminus (J-j) \in \mathcal{F}$ and note that
$|X \Delta Y' | < |X \Delta Y |$ by $|J| \geq 2$.
Then $(X,Y')$ must satisfy {\rm (B$\sp{\natural}$-EXC$_{\pm}$)}.
Since $i \in X \setminus Y'$, 
{\rm (B$\sp{\natural}$-EXC$_{\pm}$)} for $(X,Y',i)$
shows that
(a-i)
$ X - i \in \mathcal{F}$ 
\ or \ (a-ii) 
$ X - i + k  \in \mathcal{F}$
for some $k \in Y' \setminus X$.
Since $k \in Y' \setminus X \subseteq Y \setminus X$,
this means that (a) holds for $(X,Y,i)$.
A contradiction to (F).
\end{itemize}
In all cases we have reached a contradiction, which was caused by 
our initial assumption that {\rm (B$\sp{\natural}$-EXC$_{\pm}$)}
fails for some $(X,Y)$.
Therefore, {\rm (B$\sp{\natural}$-EXC$_{\pm}$)} must be true for all $(X,Y)$.
\end{proof}

\subsection{Proof of Proposition \ref{PRmultTOusualDomCond}}
\label{SCmnatmTOmnat0}

To prove
``{\rm (M$\sp{\natural}$-EXC$_{\rm m}$)} $\Longrightarrow$ {\rm (M$\sp{\natural}$-EXC)}''
under the assumption of {\rm (B$\sp{\natural}$-EXC)} for $\dom f$,
we use the following local characterization  
of M$\sp{\natural}$-concavity
(\cite{Mdcasiam},  \cite{MS99gp}).

\begin{theorem}\label{THmconcavlocexc}
A set function  $f: 2\sp{N} \to \RR \cup \{ -\infty \}$ 
is M$\sp{\natural}$-concave
if and only if 
$\dom f$ satisfies {\rm (B$\sp{\natural}$-EXC)} and 
$f$ satisfies the following three conditions:
\begin{align} 
&
f( X + i + j ) + f( X ) \leq f(X + i) + f(X + j) 
\qquad 
(\forall X \subseteq N, \ \forall i,j \in N \setminus X, \  i \not= j),
\label{mnatconcavexc1loc}
\\
&
 f( X + i + j ) + f( X + k)  \leq 
\max\left[ f(X + i + k) + f(X + j),  f(X + j + k) + f(X + i) \right]  
\notag \\
& \hspace{0.5\textwidth}
 (\forall X \subseteq N, \ \forall i,j,k \, \mbox{\rm (distinct}) \in N \setminus X) ,
\label{mnatconcavexc2loc}
\\
&
 f( X + i + j ) + f( X + k + l)  
\leq 
\max\left[  f(X + i + k) + f(X + j +l ),  f(X + j + k) + f(X + i + l) \right]  
\notag \\
& \hspace{0.5\textwidth}
 (\forall X \subseteq N, \ \forall i,j,k,l \, \mbox{\rm (distinct}) \in N \setminus X).
\label{mnatconcavexc3loc}
\end{align}
\end{theorem}

We shall derive 
 (\ref{mnatconcavexc1loc}), 
 (\ref{mnatconcavexc2loc}), and 
 (\ref{mnatconcavexc3loc}) 
from (M$\sp{\natural}$-EXC$_{\rm m}$).
First,
 (\ref{mnatconcavexc1loc}) follows from
(M$\sp{\natural}$-EXC$_{\rm m}$) applied to $(X+i+j, X, i)$,
where $J=\emptyset$ is the unique possibility.
Second, 
 (\ref{mnatconcavexc2loc}) follows from
(M$\sp{\natural}$-EXC$_{\rm m}$) applied to $(X+i+j, X+k, i)$,
where $J=\emptyset$ or $J= \{ k \}$ is possible.

To derive  (\ref{mnatconcavexc3loc}) we introduce notation
(M$\sp{\natural}$-EXC$_{\rm m}$$ (X,Y,I)$) to mean 
(M$\sp{\natural}$-EXC$_{\rm m}$) for $(X,Y,I)$, i.e., 
\begin{quote}
\mbox{\bf (M$\sp{\natural}$-EXC$_{\rm m}$$ (X,Y,I)$)}
For any $X, Y \subseteq N$ and $I \subseteq X \setminus Y$,
there exists $J \subseteq Y \setminus X$ such that
$f( X) + f( Y )   \leq  f((X \setminus I) \cup J) +f((Y \setminus J) \cup I) $
\end{quote}
with abbreviation of 
(M$\sp{\natural}$-EXC$_{\rm m}$$(X+i+j, X+k, i)$) 
to EXC$_{\rm m}$$(ij, k, i)$, etc.
We also use short-hand notation  $g(i)=f( X + i)$, 
$g(ij)=f( X + i + j )$ and $g(ijk)=f( X + i + j + k )$.
In (\ref{mnatconcavexc3loc}) we may assume $i=1$, $j=2$, $k=3$, $l=4$.
To prove
(\ref{mnatconcavexc3loc})
by contradiction, suppose that
(\ref{mnatconcavexc3loc}) fails, i.e.,
\begin{align}
\max\left[  g(13) + g(24),  g(23) + g(14) \right]  
<  g(12) + g(34) .
\label{mnatconcavexc3locFail}
\end{align}
Under the assumption (\ref{mnatconcavexc3locFail}), 
EXC$_{\rm m}$$(12, 34, 2)$ yields
\begin{align}
 g(12) + g(34) \leq \max\left[  g(1) + g(234),  g(2) + g(134) \right]  .
\label{1234leq}
\end{align}
Without loss of generality, we may assume
\[ 
 \underline{  g(2) + g(134) } \leq g(1) + g(234).
\]
EXC$_{\rm m}$$(234, 1, 4)$ yields
\[
 g(1) + g(234) \leq \max\left[  g(23) + g(14),  g(4) + g(123) \right]  
 =   g(4) + g(123).
\]
EXC$_{\rm m}$$(123, 4, 2)$ yields
\[
 g(4) + g(123) \leq \max\left[  g(24) + g(13),  g(2) + g(134) \right]  
 = \underline{  g(2) + g(134) }.
\]
Therefore,
\[
g(2) + g(134) = g(1) + g(234) = g(4) + g(123).
\]
Since we have symmetry $3 \leftrightarrow 4$, we obtain
\begin{equation}
g(1) + g(234) =  
g(2) + g(134) = 
g(3) + g(124) = 
g(4) + g(123) =: \alpha .
\label{mnatconcavexc1234alpha}
\end{equation}
By (\ref{mnatconcavexc3locFail}) and  (\ref{1234leq}) we have
\begin{align}
\max\left[  g(13) + g(24),  g(23) + g(14) \right]  
<  g(12) + g(34) \leq \alpha .
\label{mnatconcavexc3locFail2}
\end{align}
EXC$_{\rm m}$$(123, 1, 3)$, 
EXC$_{\rm m}$$(124, 2, 4)$, 
EXC$_{\rm m}$$(234, 3, 4)$, and
EXC$_{\rm m}$$(134, 4, 3)$
 yield, respectively,
\begin{align*}
  g(1) + g(123) &\leq   g(13) + g(12),
\\
  g(2) + g(124) &\leq   g(12) + g(24),
\\
  g(3) + g(234) &\leq   g(34) + g(23),
\\
  g(4) + g(134) &\leq   g(14) + g(34).
\end{align*}
By adding these four inequalities
and using (\ref{mnatconcavexc1234alpha})
we obtain
\begin{equation} \label{4a3matching}
 4 \alpha \leq 2 [g(12) + g(34) ] 
 + [ g(13) + g(24) ]  +  [ g(14) + g(23) ].
\end{equation}
This contradicts 
(\ref{mnatconcavexc3locFail2}).
Thus (\ref{mnatconcavexc3loc}) is proved.

\section*{Acknowledgements}

The author thanks Akiyoshi Shioura, Akihisa Tamura and Yu Yokoi 
for discussion and comments.
This work was supported by The Mitsubishi Foundation,   CREST, JST, 
and JSPS KAKENHI Grant Number 26280004.

\end{document}